\documentclass[11pt,a4]{amsart} 
\usepackage[hmargin=3.5cm,vmargin=3.5cm]{geometry}

\usepackage{url}
\usepackage{comment}
\usepackage{mathtools}						
\usepackage{enumerate}						
\usepackage{amsthm,amsmath,amsfonts,amsthm,amssymb}
\usepackage[all]{xy}
\usepackage{tikz}
\usepackage[foot]{amsaddr} 
\usepackage{ stmaryrd }
\usepackage{marginnote}
\usepackage{array} 
\usepackage[mathscr]{euscript}						 	%
\numberwithin{equation}{section}

\theoremstyle{plain}
\newtheorem{theorem}[equation]{Theorem}
\newtheorem*{theorem*}{Theorem}
\newtheorem{lemma}[equation]{Lemma}
\newtheorem{proposition}[equation]{Proposition}

\theoremstyle{remark}
\newtheorem{remark}[equation]{Remark}
\theoremstyle{definition}

\newtheorem*{conjecture*}{Conjecture}

\theoremstyle{definition}

\DeclareMathOperator{\divis}{div}
\DeclareMathOperator{\Div}{Div}
\newcommand{\Cl}{\textnormal{Cl}}

\newcommand{\ZZ}{{\mathbb Z}}
\author{Magnus Carlson}
\email{magnus.carlson@math.su.se}
\address{M.C: Matematiska institutionen, Stockholm University,
106 91 Stockholm, Sweden}

\author{Hee-Joong Chung}
\email{hjchung@jejunu.ac.kr}
\address{H.-J. C.: Department of Science Education, Jeju National University, 102 Jejudaehak-ro, Jeju-si, Jeju-do, 63243, South Korea}

\author{Dohyeong Kim}
\email{dohyeongkim@snu.ac.kr}
\address{D.K: Department of Mathematical Sciences and Research Institute of Mathematics, Seoul National University, 1 Gwanak-ro, Gwanak-gu, Seoul 08826, South Korea}

\author{Minhyong Kim}
\email{minhyong.kim@icms.org.uk}
\address{M.K.: International Centre for Mathematical Sciences, 47 Potterrow,  Edinburgh EH8 9BT\\
Korea Institute for Advanced Study, 85 Hoegiro, Dongdaemungu, Seoul, South Korea}

\author{Jeehoon Park}
\email{jpark.math@gmail.com}
\address{J.P: QSMS, Seoul National University, 1 Gwanak-ro, Gwanak-gu, Seoul 08826, South Korea}

\author{Hwajong Yoo}
\email{hwajong@snu.ac.kr}
\address{H.Y.: College of Liberal Studies and Research Institute of Mathematics,
Seoul National University, 1 Gwanak-ro, Gwanak-gu,  Seoul 08826, South Korea
}

\def\UX{\O_X^{\times}}

\def\<{\langle }
\def\>{\rangle}

\def\n2Z{\frac{1}{n^2}\Z/\Z}

\def\g{\gamma}
\def\k{\kappa}

\def\G{\Gamma}

\def\O{\mathcal{O}}
\def\Z{\mathbb{Z}}
\def\Spec{\textnormal{Spec}}
\def\ra{\rightarrow}
\def\A{\mathcal{A}}
\def\Q{\mathbb{Q}}
\def\C{\mathbb{C}}

\def\Gal{\textnormal{Gal}}

\def\k{\kappa}

\begin{document}
\title{Path Integrals and $p$-adic $L$-functions }


\maketitle
\begin{abstract}
We prove an arithmetic path integral formula for the inverse $p$-adic absolute values of Kubota-Leopoldt $p$-adic $L$-functions at roots of unity.
\end{abstract}

\section{Primes, knots, and quantum fields}
\subsection{Arithmetic topology}
Barry Mazur \cite{mazur1, mazur2} pointed out long ago that the cohomological properties of $\Spec(\O_F)$, the spectrum of the ring of integers of an algebraic number field, are like those of a 3-manifold. This went with the observation  that the inclusion
$$\Spec(k(P))\hookrightarrow \Spec(\O_F)$$
of the spectrum of the residue field $k(P)$ of a prime $P$ of $\O_F$ compares well to the inclusion of a knot $\k$ into a 3-manifold \cite{morishita}. When we remove a finite collection $S$ of primes and consider
$\Spec(O_F)\setminus S$, the properties then are like a 3-manifold with boundary obtained by removing  tubular neighbourhoods of  the knots. Mazur went on to consider the cases of 
$$\Spec(\Z), ~~\Spec(\Z[1/p]),  ~~ \textnormal{and} ~~
\Spec(\Z[\mu_p][\frac{1}{\zeta_p-1}]).$$ There, one arrives at an analogy between the covering
$$\Spec(\Z[\mu_{p^\infty}][1/p]) \ra \Spec(\Z[\mu_p][\frac{1}{\zeta_p-1}])$$
with Galois group $$\G:=\Gal( \Q(\mu_{p^{\infty}})/\Q(\mu_{p}))\simeq \Z_p$$
and the maximal abelian covering
$$D_{\k}\ra S^3\setminus \k,$$
which has group of deck transformations isomorphic to $\Z$.
In Iwasawa theory, the main object of study is
$$V=\Gal(M/\Q(\mu_{p^{\infty}})),$$
the Galois group of the maximal abelian unramified $p$-extension $M$ of $\Q(\mu_{p^{\infty}})$, acted on by the Iwasawa algebra
$$\Z_p[[\G]]\simeq \Z_p[[T]].$$ The isomorphism here comes from
$\g-1\mapsto T$ for a fixed topological generator $\g$ of $\G$. There is also an action of $\Gal(\Q(\mu_p)/\Q)$ (which can  be realised as a subgroup of $\Gal(\Q(\mu_{p^{\infty}})/\Q)$), according to which $V$ splits into isotypic components $V_k$ via the Teichm\"uller character $\omega: \Gal(\Q(\mu_p)/\Q)\ra \Z_p^{\times}$ and its powers $\omega^k$.
The {\em main conjecture of Iwasawa theory} \cite{mazur-wiles} as proved by Mazur and Wiles relates the determinants of these isotypic components to various branches of the $p$-adic zeta function. (For general background, we refer the reader to Washington's excellent book \cite{washington}.) Namely,  for an odd prime $p$ and for $j=1, 3, \ldots, p-2$ odd, the Kubota-Leopoldt $p$-adic $L$-function $L_p(\omega^j, s)$ is the continuous function on $\Z_p$ such that
$$L_p(\omega^j, \ell)=(1-p^{-{\ell}})\zeta(\ell) $$
for negative integers $\ell\equiv j \mod p-1$. There is a unique power series $z_j(T)\in \Z_p[[T]]$ such that
$$z_j((1+p)^s-1)=L_p(\omega^j, s),$$
enabling us to identify $z_j(T)$ with the $p$-adic $L$-function itself. The main conjecture  says that 
$z_{1-k}(T)$ is the determinant of the $\Z_p[[T]]$-module $V_k$ for $k\neq 1$ odd.
Mazur's main observation in \cite{mazur1} was that the Alexander polynomial of a knot also has a precise definition as a determinant of the module 
$$H_1(D_{\kappa}, \Z)$$
for
$$\Z[\Z]\simeq \Z[t, t^{-1}],$$
strengthening the circle of analogies that has now come to be known as {\em arithmetic topology}.

\subsection{Quantum field theory and knots invariants}
Meanwhile, in the late 1980s, Edward Witten \cite{witten} gave a remarkable construction of the Jones polynomial of a knot using the methods of quantum field theory, which were then made rigorous by Reshetikhin and Turaev \cite{RT}. Here, we have a space
$\A$ of $SU(2)$ connections 
on $S^3$ acted upon by a group $\mathcal{G}$ of gauge transformations.  The knot $\k$  defines  a Wilson loop function
$$W_{\k}: \A\ra \C$$
that sends a connection $A$ to $$\textnormal{Tr}(\rho(\textnormal{Hol}_{\k}(A))),$$ the trace of the holonomy of the connection around $\k$ evaluated in the standard representation $\rho$ of $SU(2)$. (The importance of such a function should not be surprising at all to number-theorists.) There is also a `global' {\em Chern-Simons} function given by
$$CS(A)=\frac{1}{8\pi^2}\int \textnormal{Tr}(A\wedge dA+\frac{2}{3}A\wedge A\wedge A),$$
which is only gauge-invariant up to integers. Witten's result is then
$$\int_{\A/\mathcal{G}}W_{\k}(A)  \exp(2\pi i n CS(A))d A= J_{\k}\left(\exp\left(\frac{2\pi i}{n+2}\right)\right),$$
equating a path integral with the value of the {\em Jones polynomial } $J_{\k}$ of $\k$ at a root of unity. (The reader should beware that many different normalisations exist in the literature.)  A somewhat complicated analogue for the Alexander polynomial can be found in \cite{bar-natan, cattaneo, CCFM, CCFM2, melvin-morton, rozansky}.
\subsection{Arithmetic path integrals}
We would like to prove a simple arithmetic analogue of Witten's formula for $p$-adic $L$-functions. Following the framework of {\em arithmetic topological quantum field theory}
set up in \cite{kim1, kim2, CKKPY, CKKPPY, HKM, pappas}, our goal is to represent the $p$-adic $L$-function as an arithmetic path integral, thereby incorporating the perspective of topological quantum field theory into arithmetic topology in a rather concrete fashion and strengthening the analogy envisioned by Mazur. It should be admitted right away that we do not achieve this goal. However, we do find a result about its $p$-adic valuation that appears to be interesting. To describe this, we go on to define the relevant space of `arithmetic fields'.

Let $q = p^n$ where $p$ is an odd prime and $n$ is a positive integer.
We set $K = \mathbb Q(\mu_q)$ and let $$X = \mathrm{Spec} \left( \mathbb Z[\zeta_q] \right) \backslash (\zeta_q-1),$$ where $\zeta_q$ is a primitive $q$-th root of unity.
We fix an integer $m \ge 1$ and define the space of fields as 
	\begin{align*}
		\mathscr F^m:=H^1( X,\mu_{p^m}) \times H^1_c( X, \mathbb Z / p^m \mathbb Z),
	\end{align*}
where $H^1_c$ denotes compactly supported \'etale cohomology \cite[Chapter 2]{milne}. This is an abelian moduli space of principals bundles together with its dual, a setting that allows the definition of topological actions in physics in arbitrary dimension and even on non-orientable manifolds \cite{CCFM}. 
The BF-action is the map
	\begin{align*}
		B\!F \colon \mathscr F^m \to \frac{1}{p^m} \mathbb Z / \mathbb Z
	\end{align*}
that takes $(a,b) \in \mathscr F^m$ to $$\mathrm{inv} (d a\cup b),$$ where $$d \colon H^1(X, \mu_{p^m}) \to H^2( X, \mu_{p^m})$$ is the Bockstein map coming from the exact sequence
	\begin{align*}
		1 \to \mu_{p^m} \to \mu_{p^{2m}} \to \mu_{p^m} \to 1
	\end{align*}
and $$\mathrm{inv} \colon H^3_c ( X, \mu_{p^m}) \xrightarrow{\sim} \frac{1}{p^m} \mathbb Z/ \mathbb Z$$ is the invariant map \cite{mazur2}.

	There is a natural action of $G = \mathrm{Gal}(K/\mathbb Q)$ on the space of fields $\mathscr F^m$, and we let $G' (\simeq \mathrm{Gal}(\mathbb Q(\mu_p)/\mathbb Q)) \subset G$ be the unique subgroup of $G$ of order $p-1$.
Since $p-1$ is not divisible by $p$, $G'$ acts semi-simply on $\mathscr F^m.$ Let us define $$\mathscr F^m_k:= H^1(X,\mu_{p^m})_{k} \times H^1_c(X,\ZZ/p^m\ZZ)_{-k},$$ i.e. on the first factor we take the $\omega^{k}$-eigenspace, and on the second factor we take the $\omega^{-k}$-eigenspace.

\begin{theorem} 
Let $k$ be odd and different from 1. Then we have
\begin{equation*}
|\prod_{j=0}^{p^n-1}z_{1-k}(\exp(2\pi i j/p^n)-1)^{-1}|_p=
\lim_{m\ra \infty} 
\sum_{(a,b) \in \mathscr F^m_k } \exp \left( 2 \pi i B\!F(a,b)\right)
\end{equation*}
\end{theorem}

\begin{remark}
Note that the variable $T$ here corresponds to $\gamma-1$, where $\gamma $ is a topological generator of $\G$. Thus, the value of $T$ on $\exp(2\pi i j/p^n)-1$ corresponds to the value $\exp(2\pi i j/p^n)$ assigned to $\gamma$. In the Alexander polynomial, the variable $t$ indeed corresponds to a generator of $H_1(S^3\setminus K, \Z)$. Thus, in the analogy of arithmetic topology, the values occurring on the left of our formula correspond exactly to the values of the standard variable in topology at roots of unity. We could use the Weierstrass preparation theorem  and the vanishing of the $\mu$-invariant \cite{FW} to replace $z_j$ by a distinguished polynomial $P_j$, such that $$z_j(T)=P_j(T)u_j(T)$$
for a unit $u_j(T)$. With  this we can restore the variable $t=T+1$ and put $Q_j(t)=P_j(t-1)$.
The left-hand side of our formula can then be written
$$|\prod_{j=0}^{p^n-1}Q_{1-k}(\exp(2\pi i j/p^n))^{-1}|_p.$$
One might argue that the $Q_j$ are the true analogues of the Alexander polynomial.
\end{remark}
\begin{remark}
Of course since we are taking the $p$-adic absolute value, the formula is the same if we change $z_k$ by a unit. Hence, it really is about  characteristic power series rather than the precise choice of $p$-adic $L$-functions. Nevertheless, since the $p$-adic $L$-functions are the objects of central interest in number theory, we have stated the theorem in these terms. Obviously, for this we need the main conjecture of Iwasawa theory  which we will take for granted  in the rest of this paper.
\end{remark}
\begin{remark}
In physics, it's common to be vague about the domain of the path integral. That is, one imagines a sequence of inclusions
$$\mathcal{C}\subset \A_1\subset \A_2\subset \cdots $$
containing the space $\mathcal{C}$ of classical fields (solutions to the equation of motion) and integrals over $\A_n$ giving successively more information. On the other hand, after some point, further enlargement shouldn't matter. That is, the inclusion  $\mathcal{C}\subset \A_n$ into a sufficiently flabby space should play a role similar to an acyclic resolution of a complex where any two resolutions are suitably homotopic. In gauge theory, for example, the space of $C^{\infty}$ connections is thought to be an adequate domain. Even there, one could include discontinuous or distributional connections in the flavour of Feynman's original heuristic arguments with jagged paths. The limit we are taking can be interpreted as
$$\int_{\mathscr F_k} \exp(2\pi iBF(a,b)) dadb,$$
an integral over the domain
\begin{equation*}
\begin{split}
\mathscr F_k&= \varprojlim_m H^1(X,\mu_{p^m})_{k} \times \varinjlim_m H^1_c(X,\ZZ/p^m\ZZ)_{-k}\\
&=H^1(X, \Z_p(1))_k\times H^1_c(X, \Q_p/\Z_p)_{-k}.
\end{split}
\end{equation*}
This has very  much the flavour of a space of distributional fields.

\end{remark}
\begin{remark}
The obvious challenge is to remove the absolute value from the $p$-adic $L$-value, incorporating the unit information. In the analogy with physics, a unit is like a `phase', leading one to believe that it can be recovered from a refinement of the given path integral.
\end{remark}

\medskip
\section{Path integrals for cyclotomic integers}
Let $\textnormal{Cl}_K$ be the ideal class group of $K$ and $\O_X^\times$ the group of units in $\Z[\mu_q][1/(\zeta_q-1)].$ A repetition of the  computations  found in \cite{carlson-kim} shows that the arithmetic path integral
	\begin{align*}
		\sum_{(a,b) \in \mathscr F^m } \exp \left( 2 \pi i B\!F(a,b)\right)
	\end{align*}
equals
	\begin{align*}
		\left| p^m \cdot \Cl_K[p^{2m}]\right| 
		\cdot 
		\left | \O_X^\times / \left(\O_X^\times\right)^{p^m}\right| 
		\cdot 
		\left | \Cl_K / p^m\right |
		.
	\end{align*}
	Our goal will be to prove an equivariant version of this formula. 
	
	With the definitions of the previous section, we have that
	\begin{align*}
		\mathscr F^m = \bigoplus_{k=0}^{p-2} \mathscr F^m_k.
	\end{align*}
We will generally denote by a subscript $(\cdot)_k$ the $\omega^k$-isotypic component of a $\Z_p$-module with $G'$-action. We now analyze how the  action of $G'$ on $\mathscr{F}^m$
interacts with the BF-functional. More precisely, we will see that the BF-functional splits: $$ \oplus_{k=0}^{p-2} B\!F_k:  \bigoplus_{k=0}^{p-2} \mathscr F^m_k \rightarrow \dfrac{1}{p^m} \ZZ/ \ZZ.$$ 
To see this, start by letting $\Div X$ be the free abelian group generated by the closed points in $X.$ Then, as is explained in \cite[Section 2]{carlson-kim}, we have that 
\begin{equation} H^i(X,\mu_{p^m}) = \begin{cases} \mu_{p^m}(K) & \text{for } i=0,  \\ Z_1/B_1 & \text{for } i = 1, \\ \textnormal{Cl}_K/p^m& \text{for } i=2, \\ 0 & \text{for } i >2, \end{cases} \end{equation} 
where 
$$Z_1 = \{(a,I) \in K^* \oplus \Div X : \divis(a)+p^mI=0 \}$$ and $$B_1 = \{(a^{p^m},- \divis(a)) \in K^* \oplus \Div X:  a \in K^* \}.$$ 
As is shown in \cite[Section 4]{alhqvist}, the map $d: H^1(X,\mu_{p^m}) \rightarrow H^2(X,\mu_{p^m})$ takes $$(a,I) \in H^1(X,\mu_{p^m})$$ to $I \in \Cl_K /p^m$.
By this observation, it is clear that the map $d$ is equivariant. Let us now note further that the Galois action on 
\begin{equation*}
H^3_c(X,\mu_{p^m}) \underset{\mathrm{inv}}{\simeq} \dfrac{1}{p^m} \ZZ/\ZZ
\end{equation*}
is trivial. 
This clearly implies that the BF-functional splits into direct sums as claimed, since if $$a \in H^1(X,\mu_{p^m})_{i} \quad \text{ and } \quad b \in H^1_c(X,\ZZ/{p^m}\ZZ)_{j},$$ we see that $da \cup b$ lands in the $\omega^{i+j}$-eigenspace of $H^3_c(X,\mu_{p^m})$, which is non-zero if and only if $i+j=0 \pmod{p-1}.$ 

By the above analysis of the $G'$-action, we see that the sum
	\begin{align*}
		\sum_{(a,b) \in \mathscr{F}^m} \exp \left(2 \pi i B\!F (a,b)\right)
	\end{align*}
splits:
	\begin{align*}
		\prod_{k=0}^{p-2} \sum_{(a,b) \in  \mathscr{F}^m_k}  \exp \left(2 \pi i B\!F (a,b)\right).
	\end{align*}
Using the description of the map $d$ above (in particular that it is equivariant), together with the non-degeneracy of the Artin--Verdier pairing, we find that
\begin{proposition}
	\begin{align*}
		\sum_{(a,b) \in \mathscr F^m_k}  \exp \left(2 \pi i B\!F (a,b)\right)
		=
		\left| \left(p^m \cdot \Cl_K[p^{2m}]\right)_{k}\right| 
		\cdot 
		\left |\left( \O_X^\times / \left(\O_X^\times\right)^{p^m}\right)_{k} \right| 
		\cdot 
		\left | \left(\Cl_K / p^m\right)_k\right |
		.
	\end{align*}
	\end{proposition}
\begin{proof}
If $a \not\in \ker d$, then 
\begin{equation*}
\sum_{b \in H^1_c(X,\ZZ/p^m\ZZ)_{-k}}\exp \left(2 \pi i B\!F (a,b)\right)=\sum_{b \in H^1_c(X,\ZZ/p^m\ZZ)_{-k}}\exp\left(2\pi i \cdot \mathrm{inv}(da \cup b)\right)=0
\end{equation*}
because the Artin--Verdier pairing is non-degenerate. Since $\exp\left(2\pi i B\!F (a,b)\right)=1$ for any $a \in \ker d$, we have
\begin{equation*}
\sum_{(a,b) \in  \mathscr{F}^m_k}  \exp (2 \pi i B\!F (a,b))=|\ker d \cap H^1(X, \mu_{p^m})_k| \cdot |H^1_c(X,\ZZ/p^m\ZZ)_{-k}|.
\end{equation*}
Also, we have 
\begin{equation*}
(\ker d)_k=\ker d \cap H^1(X, \mu_{p^m})_k
\end{equation*}
since the map $d$ is equivariant. Furthermore, $d$ is the composite of two maps: the surjective map $f_1:H^1(X,\mu_{p^m}) \rightarrow \Cl_K[p^m],$ and then the reduction map $$f_2: \Cl_K[p^m] \rightarrow  \Cl_K/p^m,$$ which are both equivariant as well.  The kernel of $f_1$ is precisely $\mathcal{O}^\times_X/p^m,$ while the kernel of $f_2$ equals $p^m \cdot \Cl_K[p^{2m}].$ By taking eigenspaces, we have $$|(\ker d)_k|=|(\mathcal{O}^\times_X/p^m)_k| \cdot |(p^m \cdot \Cl_K[p^{2m} ])_k|.$$ 
Finally, $H^1_c(X,\ZZ/p^m)_{-k}$ is dual to $H^2(X,\mu_{p^m})_k$ and so the result follows.
\end{proof}

It is interesting to realize the path integral
	\begin{align}\label{eq:integral-k}
		\sum_{(a,b) \in \mathscr F^m_k}  \exp \left(2 \pi i B\!F (a,b)\right)
	\end{align}
as a path integral on $$Y := \mathrm{Spec} \left( \mathbb Z \right) \backslash \{p\} .$$
To achieve this, we note that the natural map $\pi \colon X \to Y$, factors as $$X \xrightarrow{\pi'} Y' \xrightarrow{\pi''} Y$$ where $$Y' = \mathrm{Spec}(\mathbb Z[\mu_p])\backslash (\zeta_p-1).$$
If we consider the sheaf $\pi_* (\mathbb Z/p^m\mathbb Z)$, it is easy to see that this sheaf corresponds, under the equivalence between sheaves on $Y$ split by $\pi$ and $G$-modules, to the group ring
	\begin{align*}
		\mathbb Z/p^m \mathbb Z[G] 
		\cong
		\mathbb Z / p^m \mathbb Z [x,y] / \left(x^{p^{m-1}}-1, y^{p-1}-1\right)
		.
	\end{align*}
This group ring is isomorphic, as a $G$-module, to
	\begin{align*}
		\bigoplus_{k=0}^{p-2} \left( \mathbb Z / p^m \mathbb Z [x] / \left(x ^{p^{m-1}}-1\right)\right)_{k}
	\end{align*}
where the action of $G'$ on the $k$-th piece in the direct sum is through $\omega^k$.
This calculation shows that $\pi_*(\mathbb Z/ p^m \mathbb Z)$ splits into a direct sum
	\begin{align*}
		\bigoplus_{k=0}^{p-2} \mathscr M_k
	\end{align*}
of ``eigensheaves'' with respect to the $G'$-action. Since Cartier duality commutes with pushforward \cite[Proposition D.1]{rosengarten}, we see that
	\begin{align*}
		\pi_* (\mu_{p^m}) = \bigoplus_{k=0}^{p-2} D(\mathscr M_k),
	\end{align*}
where $D$ denotes the Cartier dual.
For $k=0,1,\cdots,p-2$, we now claim that
	\begin{align*}
		\mathscr F^m_k = H^1(Y, D(\mathscr M_k)) \times H^1_c(Y, \mathscr M_k).
	\end{align*}
To see this, it is enough to establish that $H^1(Y, D(\mathscr M_k))$ identifies with the $\omega^k$-eigenspace of $H^1(X,\mu_{p^m})$ and that  $H^1_c(Y, \mathscr M_k)$ naturally identifies with the $\omega^{-k}$-eigenspace of $H^1_c(X,\ZZ/p^m\ZZ)$.

We proceed by first noting that $\pi_*(\mu_{p^m}) = \left( \pi_*'' \circ \pi_*' \right) \mu_{p^m}$, and that since $|G'|$ is prime to $p$, taking $G'$-fixed points is an exact functor.  Then we have the following string of equalities:
	\begin{align*}
		H^0(G' , H^i(X, \mu_{p^m }))
		=
		H^0 (G' , H^i(Y',\pi_*' \mu_{p^m}))
		=
		H^0(G' , H^i(Y',D(\mathscr M_0)))
		=
		H^i(Y , D(\mathscr M_0 )).
	\end{align*}
	The first equality follows from the fact that $\pi'$ is finite \'etale, the second follows from the fact that the restriction of $D(\mathscr M_0)$ to $Y'$ is isomorphic to $\pi_* ' \mu_{p^m}$, and the last follows from the fact that taking $G'$-fixed points is an exact functor.   
This shows that $H^i(Y , D(\mathscr M_0))$ naturally identifies with the $\omega^0$-eigenspace of $H^i (X, \mu_{p^m }),$ and we proceed by analysing the other eigenspaces. Let us note that 
$$D(\mathscr M_k)= D(\mathscr M_0) \otimes \mathscr G_k,$$ 
where $\mathscr G_k$ is the sheaf which, under the equivalence between sheaves split by $\pi''$ and $G'$-modules, corresponds to $\mathbb Z/p^m\mathbb Z$, but where the action of $G'$ is through $\omega ^{-k}$. We now claim that $$H^i(Y',D(\mathscr M_k)) = H^i(Y',D(\mathscr M_0))(\omega^{-k})$$ where $H^i(Y,D(\mathscr M_0))(\omega^{-k})$ is just $H^i(Y',D(\mathscr M_0))$ as an abelian group, but with the $G'$-action twisted by $\omega^{-k}$. Indeed, the pullback of $$D(\mathscr M_0) \otimes \mathscr G_k$$  to $Y'$ is isomorphic, as an abelian sheaf, to $D(\mathscr M_0),$ and looking at the Cech complex one finds that $H^i(Y',D(\mathscr M_0))(\omega^{-k})$ is indeed isomorphic to $H^i(Y',D(\mathscr M_k)).$ 

Thus, \begin{align*}
		H^0(G' , H^i(X, \mu_{p^m })(\omega^{-k}))
		=
		H^0(G' , H^i(Y',\pi_*' \mu_{p^m})(\omega^{-k}))
		=
		H^0(G' , H^i(Y',D(M_0))(\omega^{-k}))
	\end{align*}
which equals $H^0(G',H^i( Y', D(\mathscr M_k ))) = H^i(Y, D(\mathscr M_k)).$ Since $$H^i(X,\mu_{p^m})_k = H^0( G' , H^i( X, \mu_{p^m })(\omega^{-k}))$$  this gives that $H^i(X,\mu_{p^m})_k = H^i(Y,D(\mathscr M_k))$, which is the equality we wished to prove. Repeating  the same argument for $\mathscr M_k$ then shows that
	\begin{align*}	
		\mathscr F^m_k =  H^1( Y, D(\mathscr M_k)) \times H^1_c( Y, \mathscr M_k),
	\end{align*}
as claimed.  

The realization of the path integral \eqref{eq:integral-k} as a path integral on $Y$ is now straightforward.
Indeed, there is a natural Bockstein map $d \colon H^1( Y, D(\mathscr M_k)) \to H^2 (Y, D(\mathscr M_k))$ and we define 
\begin{align*}
B\!F \colon \mathscr F^m_k = H^1(Y,D(\mathscr M_k)) \times H^1_c(Y, \mathscr M_k) \to \frac{1}{p^m} \mathbb Z / \mathbb Z
\end{align*}
as $B\!F (a,b) = \mathrm{inv}(d a \cup b )$.
This realizes the path integral as a path integral on $Y$, which is what we wanted to see.

\medskip
\section{Calculation for large values of $m$}
We analyse the right-hand-side of the formula in Proposition 2.1 for large values of $m$.
If $p^m> \left|\Cl_K\right|$, then the first factor is one and the third factor equals the size of the $\omega^k$-isotypic component of the $p$-primary part of $\Cl_K$.
Below, we look into the factor in the middle. We will assume for simplicity that 
 $p^m > |\Cl_K[p^{\infty}]|$.

Let $U'\subset \O_X^\times$ be the subgroup generated by a primitive $q$-th root of unity $\zeta$ and elements of the form $1 - \zeta ^a$ where $a=q, 2,3,\cdots,q-1$. 
Therefore, we have a finite quotient $A:= \UX /U'$.

From the exact sequence
$$0\rightarrow U'\rightarrow \O_X^\times\ra A\ra 0$$
and the snake lemma, we get the long exact sequence
$$0\ra U'[p^m]\ra \O_X^\times [p^m]\ra A[p^m]$$
$$\ra U'/(U')^{p^m}\ra \UX/(\UX)^{p^m}\ra A/A^{p^m}\ra 0.$$
Since
$$U'[p^m]\simeq \O_X^\times [p^m]\simeq \mu_{p^n},$$
for $m\geq n$,
this gives an exact sequence
$$0\ra A[p^m]\ra U'/(U')^{p^m}\ra \O_X^\times/(\O_X^\times)^{p^m}\ra A/A^{p^m}\ra 0$$
and the same after taking $\omega^k$-isotypic components:
$$0\ra A[p^m]_k\ra (U'/(U')^{p^m})_k\ra (\UX/(\UX)^{p^m})_k\ra (A/A^{p^m})_k\ra 0.$$
Note that 
$A[p^m]_k \simeq A_k[p^m]$ and $(A/A^{p^m})_k \simeq A_k/{A_k^{p^m}}$ as $|G'|$ has order prime to $p$. Since $A$ and hence $A_k$ is finite, the kernel and cokernel have the same order, so that
$$|(U'/(U')^{p^m})_k|=| (\UX/(\UX)^{p^m})_k|.$$
Put $U := U'/\mu_{q}$. Since $U$ is torsion-free, another easy snake lemma argument gives an exact sequence
$$0\ra \mu_q\ra U'/(U')^{p^m}\ra U/U^{p^m}\ra 0$$
for $m\geq n$.
Thus, for $k\neq 1$, we get an isomorphism
$$(U'/(U')^{p^m})_k\simeq (U/U^{p^m})_k. $$

\begin{lemma}
We have
\begin{equation*}
\left(U'/(U')^{p^m}\right)_k = \begin{cases}
\{ 1\} & \text{ if $k$ is odd and $k\neq 1$},\\
\mathbb Z/{p^m\mathbb Z} & \text{ if $k$ is even}.
\end{cases}
\end{equation*}
\end{lemma}
\begin{proof}
The structure of $U$ as a Galois module is known. 
Let $K^+$ be the maximal totally real subfield of $K$ and let $G^+ =\mathrm{Gal}(K^+/\mathbb Q)$.
In \cite[Theorem 3]{bass}, Bass proved that there is an isomorphism $U \simeq \mathbb Z[G^+]$ as Galois modules.
Thus, the assertion follows since $\mathbb Z/p^m\mathbb Z[G^+] = \bigoplus_{k \textrm{:even}} \left( \mathbb Z/p^m\mathbb Z[G^+]\right)_k$ and each summand is isomorphic to $\mathbb Z/p^m\mathbb Z$. 
\end{proof}

As a consequence, for even $k$, we get
\begin{align}\label{eq:formula2}
\lim_{m \to \infty} \frac{1}{p^m}\sum_{(a,b) \in \mathscr F^m_k}  \exp(2 \pi i B\!F (a,b)) = \left| \mathrm{Cl}_K[p^\infty]_k\right|
\end{align}
and for odd $k\neq 1$, 
\begin{align}\label{eq:formula3}
\lim_{m \to \infty}\sum_{(a,b) \in \mathscr F^m_k}  \exp \left(2 \pi i B\!F (a,b)\right) = \left| \mathrm{Cl}_K[p^\infty]_k\right|.
\end{align}

\medskip
\section{Connection to $p$-adic $L$-functions: Proof of Theorem 1.1.}
We interpret the right-hand side of the formula in terms of  special values of the Kubota-Leopoldt $p$-adic $L$-functions.
As in the introduction, let
$$V=\Gal(M/\Q(\mu_{p^{\infty}})),$$
the Galois group of the maximal abelian unramified $p$-extension $M$ of 
$\Q(\mu_{p^{\infty}})$. 
We now vary $n$ in $q=p^n$, and denote by $\Cl_{K_n}[p^{\infty}]$ the $p$-primary part of the ideal class group of $K_n:=\mathbb Q(\mu_{p^{n+1}})$. 
We assume in the following that $k\neq 1$ is odd. By the main conjecture and the vanishing of the $\mu$-invariant \cite{FW}, we have an exact sequence
$$0\ra A\ra V_k\ra M\ra B\ra 0,$$
where $A$ and $B$ are finite,
$$M\simeq \prod_j \Z_p[[T]]/(f_j(T) )$$
and the $f_j(T)$ are power series such that
$$z_{1-k}(T)=\prod_j f_j(T).$$
Consider the action of $\G_n\subset \G$ generated by $\g^{p^n}$, where $\g$ is a topological generator of $\G$. We get an equation of $\G_n$-Euler characteristics
$$\frac{\chi(\G_n, A)\chi(\G_n, M)}{\chi(\G_n, V_k)\chi(\G_n, B)}=1$$
where 
$$\chi(\G_n, \cdot)=\prod |H_i(\G_n, \cdot)|^{(-1)^i}.$$
Note that $$H_i(\G_n, \cdot)\simeq H^i(\G_n, (\cdot)^{\vee})^{\vee},$$
where $$(\cdot)^{\vee}=\mbox{Hom}(\cdot, \Q/\Z)$$
is the Pontriagin dual.
On the other hand, since $A$ and $B$ are finite, the exact sequence
$$0\ra H^0(\G_n, A^{\vee})\ra A\stackrel{\g^{p^n}-1}{\longrightarrow}A\ra H^1(\G_n, A^{\vee})\ra 0$$
and the similar one for $B$
imply that
$$\chi(\G_n, A)=\chi(\G_n, B)=1.$$
Therefore,
$$\chi(\G_n, V_k)=\chi(\G_n, M).$$
On the other hand, since all the extensions $K_n/\Q(\mu_p)$ are totally ramified over the single prime lying above $p$, by \cite[Section 1]{greenberg} we get
 $$(\Cl_{K_n}[p^{\infty}])_k\simeq V_k/((T+1)^{p^n}-1)V_k=H_0(\G_n, V_k).$$
By \cite[Corollary 13.29]{washington}, we have
$$V_k\simeq \Z_p^a$$
for some $a\geq 0$. We have the exact sequence
$$0\ra H_1(\G_n, V_k)\ra V_k\stackrel{\g^{p^n}-1}{\longrightarrow}V_k \ra H_0(\G_n, V_k)\ra 0.$$
Since $H_0(\G_n, V_k)$ is finite, the operator $\g^{p^n}-1$ does not have the eigenvalue 0. Hence,
$$H_1(\G_n, V_k)=0.$$
Similarly, $H_1(\G_n, M)=0$. Thus,
$$|V_k/((T+1)^{p^n}-1)V_k|=|H_0(\G_n, V_k)|=|H_0(\G_n, M)|=|M/((T+1)^{p^n}-1)M|.$$
 
 So finally,
\begin{equation*}
\begin{split}
|(\Cl_{K_n}[p^{\infty}])_k| &= |V_k/((T+1)^{p^n}-1)V_k| \\
&= |\Z_p[[T]]/( (T+1)^{p^n}-1, z_{1-k}(T))|\\
&=|\prod_j \Z_p/(z_{1-k}(\zeta_{p^n}^j-1))|\\
&=|\prod_j z_{1-k}(\zeta_{p^n}^j-1)^{-1}|_p
\end{split}
\end{equation*}
yielding the desired formula.

\begin{remark}
For even values of $k$, \eqref{eq:formula2} implies the vanishing
\begin{align}
\lim_{m \to \infty} \frac{1}{\sum_{(a,b) \in \mathscr F^m_k}  \exp \left(2 \pi i B\!F (a,b)\right)}=0,
\end{align}
which is superficially analogous to the vanishing of $L(\chi,s)$ at all negative integers when $\chi$ is an odd Dirichlet character.
\end{remark}
\section*{Acknowledgements}
Magnus Carlson would like to thank the Knut and Alice Wallenberg foundation for their support. The work of DK was supported by the National Research Foundation of Korea (2020R1C1C1A0100681913) and by Samsung Science and Technology Foundation (SSTF-BA2001-01).
M.K. was supported in part by  UKRI grant EP/V046888/1 and a Simons Fellowship at the Isaac Newton Institute. He is grateful to David Ben-Zvi, Dennis Gaitsgory, Barry Mazur, and Akshay Venkatesh for helpful conversations about arithmetic topology. The work of Jeehoon Park was supported by the National Research Foundation of Korea (NRF-2021R1A2C1006696) 
and the National Research Foundation of Korea (NRF) grant funded by the Korea government (MSIT) (No.2020R1A5A1016126). The work of H.Y. was supported by the National Research Foundation of Korea(NRF) grant funded by the Korea government(MSIT) (No. 2020R1A5A1016126).




\begin{thebibliography}{30}
\bibitem{alhqvist}{Ahlqvist, Eric; Carlson, Magnus The \'etale cohomology ring of the ring of integers of a number field, 2018.}
\bibitem{bar-natan}{Bar-Natan, D.; Garoufalidis, S. 
On the Melvin-Morton-Rozansky conjecture. Inventiones mathematicae volume 125, pages 103--133 (1996)}
\bibitem{bass}{Bass, Hyman Generators and relations for cyclotomic units. Nagoya Math. J., 27:401--407, 1966.}

\bibitem{carlson-kim}{Carlson, Magnus; Kim, Minhyong A note on abelian arithmetic bf-theory. Bulletin of the London
Mathematical Society. Published online, April, 2022. https://londmathsoc.onlinelibrary.wiley.com/doi/full/10.1112/blms.12629?af=R}
 \bibitem{CKKPY}{Chung, Hee-Joong; Kim, Dohyeong; Kim, Minhyong; Park, Jeehoon; Yoo, Hwajong Arithmetic Chern-Simons theory II. P-adic hodge theory, 81--128, Simons Symp., Springer, Cham, [2020].}
 
  \bibitem{CKKPPY}{Chung, Hee-Joong; Kim, Dohyeong; Kim, Minhyong; Pappas, George; Yoo, Hwajong
  Abelian Arithmetic Chern-Simons Theory and Arithmetic Linking Numbers. International Mathematics Research Notices, Volume 2019, Issue 18, September 2019, Pages 5674--5702
  }
  \bibitem{cattaneo}{
  Cattaneo, A.S. Cabled Wilson loops in BF theories. J. Math. Phys. 37, 3684 (1996).}

\bibitem{CCFM}
{Cattaneo, A.S.; Cotta-Ramusino, P.; Fr\"ohlich, J.; Martellini, M. Three dimensional BF theories and the Alexander-Conway invariant of knots. Neclear Physics B Volume 436, Issues 1--2, 20, February 1995, Pages 355--382.}

\bibitem{CCFM2}
{Cattaneo, A.S.; Cotta-Ramusino, P.; Fr\"ohlich, J.; Martellini, M. Topological BF theories in 3 and 4 dimensions. J. Math. Phys. 36 (1995), 6137.}

\bibitem{FW}{Ferrero, Bruce; Washington, Lawrence C.
The Iwasawa invariant $\mu_p$ vanishes for abelian number fields.
Ann. of Math. (2) 109 (1979), no. 2, 377--395.}
\bibitem{greenberg}{Greenberg, Ralph Iwasawa theory--past and present. {\em Class field theory: its centenary and prospect} (Tokyo, 1998), 335--385,
Adv. Stud. Pure Math., 30, Math. Soc. Japan, Tokyo, 2001.}

\bibitem{hart}{Hart, William; Harvey, David; Ong, Wilson Irregular primes to two billion. Math. Comp.,
86(308):3031--3049, 2017.}
\bibitem{HKM}{Hirano, Hikaru; Kim, Junhyeong; Morishita, Masanori
On arithmetic Dijkgraaf-Witten theory. arXiv:2106.02308}

\bibitem{iwasawa}{Iwasawa, Kenkichi On some modules in the theory of cyclotomic fields, J. Math. Soc. Japan (1964).}

\bibitem{kim1}
{Kim, Minhyong Arithmetic gauge theory: a brief introduction. Modern Phys. Lett. A 33 (2018), no. 29, 1830012, 26 pp. }
\bibitem{kim2}{Kim, Minhyong Arithmetic Chern-Simons theory I. Galois covers, Grothendieck-Teichmller Theory and Dessins d'Enfants, 155--180, Springer Proc. Math. Stat., 330, Springer, Cham, [2020], 2020}

\bibitem{mazur1}
{Mazur, Barry Remarks on the Alexander Polynomial.
Typewritten notes available at https://people.math.harvard.edu/~mazur/papers/alexander\_polynomial.pdf}
\bibitem{mazur2}{Mazur, Barry
Notes on \'etale cohomology of number fields.
Ann. Sci. \'Ecole Norm. Sup. (4) 6 (1973), 521--552 (1974).}

\bibitem{milne}{Milne, J.S. Arithmetic duality theorems. BookSurge, LLC, Charleston, SC, second edition, 2006.}
\bibitem{mazur-wiles}{
Mazur, B.; Wiles, A. Class fields of abelian extensions of Q. Invent. Math. 76 (1984), no. 2, 179--330. }
\bibitem{melvin-morton}{Melvin, P. M.; Morton H. R. The coloured Jones function.
Comm. Math. Phys. 169(3): 501-520 (1995).
}
\bibitem{morishita}{Morishita, Masanori Knots and primes. An introduction to arithmetic topology. Universitext. Springer, London, 2012. xii+191 pp. 

}
\bibitem{pappas}{Pappas, Georgios Volume and symplectic structure for $\ell$-adic local systems. Adv. Math. 387 (2021), Paper No. 107836, 70 pp.}

 \bibitem{RT}
 Reshetikhin, N.; Turaev, V. G. Invariants of 3-manifolds via link polynomials and quantum groups. Invent. Math. 103 (1991), no. 3, 547--597.
 
 \bibitem{rosengarten}{Rosengarten, Zev Tate duality in positive dimension over function fields. arxiv.1805.00522, 2018.}
 \bibitem{rozansky}{Rozansky, L. A contribution of the trivial connection to the Jones polynomial and Witten's invariant of 3d manifolds, I. Communications in Mathematical Physics volume 175, pages 275--296 (1996)}
 \bibitem{washington}{Washington, Lawrence C. Introduction to cyclotomic fields. Second edition. Graduate Texts in Mathematics, 83. Springer-Verlag, New York, 1997. xiv+487 pp.}

\bibitem{witten}{Witten, Edward Quantum field theory and the Jones polynomial. Comm. Math. Phys. 121 (1989), no. 3, 351--399.}
\end{thebibliography}
\end{document}